\documentclass{aptpub}

\usepackage[colorlinks=true,linkcolor=blue]{hyperref}
\hypersetup{urlcolor=blue, citecolor=blue, colorlinks=true, linkcolor=blue}
\usepackage[margin=1in]{geometry}

\usepackage{graphicx}
\usepackage{multirow}
\usepackage{rotating}
\usepackage{textcomp}
\usepackage[english]{babel}

\makeatletter
\def\varphiall{\@ifnextchar[{\varphiall@i}{\varphiall@i[]}}
\def\varphiall@i[#1]{\@ifnextchar[{\varphiall@ii{#1}}{\varphiall@ii{#1}[#1]}}
\def\varphiall@ii#1[#2]{\varphi_{\mu,\kappa,\beta_{#1},\sigma^2_{#2}}}
\def\hatvarphiall{\@ifnextchar[{\hatvarphiall@i}{\hatvarphiall@i[]}}
\def\hatvarphiall@i[#1]{\@ifnextchar[{\hatvarphiall@ii{#1}}{\hatvarphiall@ii{#1}[#1]}}
\def\hatvarphiall@ii#1[#2]{\widehat{\varphi}_{\mu,\kappa,\beta_{#1},\sigma^2_{#2}}}
\makeatother

\authornames{T.~FAOUZI \emph{ET AL}}
\shorttitle{The Dagum family of Isotropic Covariance Functions} % insert short title here for use in running head

\begin{document}%\recd{}{}%Do not alter this line.

\title{A Deep Look into the Dagum Family of Isotropic \\ Covariance Functions} % insert title - use \\ if it requires more than one line.

\authorone[University of Bio Bio]{Tarik Faouzi}
\authortwo[KU $\&$ Trinity College Dublin]{Emilio Porcu}
\authorthree[University of Bio Bio]{Igor Kondrashuk}
\authorfour[M\"{a}lardalen University]{Anatoliy Malyarenko}

%Please use the following format for addresses and emails. The APT office will sort this out after you submit your files.
\addressone{Department of Statistics, University of Bio Bio, Concepci{\'o}n, Chile} % Your postal address goes here.
\emailone{faouzi0612@gmail.com } %Authors email goes here.
\addresstwo{Department of Mathematics, Khalifa University at Abu Dhabi, $\&$ School of Computer Science and Statistics, Trinity College Dublin} % Your postal address goes here.
\emailtwo{emilio.porcu@ku.ac.ae}
\addressthree{ Grupo de Matem\'atica Aplicada, Departamento de Ciencias B\'asicas, Universidad del B\'io-B\'io,  Campus Fernando May, Av. Andres Bello 720, Casilla 447, Chill\'an, Chile.  } % Your postal address goes here.
\emailthree{igor.kondrashuk@gmail.com }
\addressfour{Division of Mathematics and Physics, M\"{a}lardalen University, Box~883, 721~23 V\"{a}ster{\aa}s, Sweden} % Your postal address goes here.
\emailfour{anatoliy.malyarenko@mdh.se }

\vspace{-5cm}

\begin{abstract}
The Dagum family of isotropic covariance functions has two parameters that  allow for decoupling of the fractal dimension and Hurst effect for Gaussian random fields that are stationary and isotropic over  Euclidean spaces.

Sufficient conditions that allow for positive definiteness in $\mathbb{R}^d$ of the Dagum family have been proposed on the basis of the fact that the Dagum family allows for complete monotonicity under some parameter restrictions.

The spectral properties of the Dagum family have been inspected to a very limited extent only, and this paper gives insight into this direction. Specifically, we study finite and asymptotic properties of the isotropic spectral density (intended as the Hankel transform) of the Dagum model. Also, we establish some closed forms expressions for the Dagum spectral density in terms of the Fox--Wright functions. Finally, we provide asymptotic properties for such a class of spectral densities.
\end{abstract}


%insert keywords separated by a semicolon. You should avoid including keywords which also appear in the title.

% insert the primary 2020 Maths Subject Classification number in the first bracket
		% and the secondary ams number(s) in the second bracket
		% e.g. \ams{60E20}{49G03;49F10}
		%Maximum of three in each, ideally one or two in each primary and secondary.
		%codes found here ``https://mathscinet.ams.org/msnhtml/msc2020.pdf''

%\newpage
\section{Introduction}

Isotropic covariance functions have a long history that can be traced back to \cite{matheron1963principles} and \cite{yaglom1957some}. The mathematical machinery that allows to implement isotropic covariance functions is based on positive definite functions, that are radially symmetric over $d$-dimensional Euclidean spaces. In particular, the characterization of the radial part of radially symmetric positive definite functions was provided in the {\em tour de force} by \cite{Shoe38}. There is a rich catalogue of isotropic covariance functions that are obtained by composing a given parametric family of functions defined on the positive real line with the Euclidean norm in a $d$-dimensional Euclidean space \cite{daley-porcu}. What makes them interesting is that some parameters have a corresponding interpretation in terms of geometric properties of Gaussian random fields. For instance, the Mat{\'e}rn family has a parameter that allows to verify the mean square differentiability of its corresponding Gaussian random field, in concert with its fractal dimension.

The Dagum parametric family of functions was originally proposed by \cite{porcu-thesis} as a new family of isotropic covariance functions associated with Gaussian random fields that are weakly stationary and isotropic over $d$-dimensional Euclidean spaces. Sufficient conditions for positive definiteness in $\mathbb{R}^3$ have been provided by \cite{porcu2007modelling} on the basis of a criterion of the P{\'o}lya type \cite{gneiting2001criteria}. Later, \cite{mateu2007note} have shown that the Dagum family allows for decoupling fractal dimension and Hurst effect, allowing to avoid self-similar random fields, and consequently all the issues that are related to the estimation of fractal dimension and long memory parameters under self similarity \cite{kent1997estimating}.

Positive definiteness of a given radial function over all $d$-dimensional Euclidean spaces is equivalent to complete monotonicity of its radial part \cite{Shoe38}. \cite{berg2008} have proved sufficient conditions for complete monotonicity of the Dagum class. Some necessary conditions have also been provided therein, but unfortunately these do not match with the sufficient ones. Hence, a complete characterization for the Dagum function is, up to date, still elusive.

A wealth of applications in applied branches of science has shown how the Dagum family can be used to model temporal or spatial phenomena where local properties (fractal dimension) and global ones (Hurst effect) are decoupled, and we refer the reader to \cite{shen2014bernoulli}, with the references therein.

Positive definiteness of a radially symmetric function in a $d$-dimensional Euclidean space, for a given dimension $d$, is equivalent to its Hankel transform, called radial spectral density, being nonnegative and integrable \cite{daley-porcu}. The radial spectral density is often not available in closed form, with the notable exception of the Mat{\'e}rn model \cite{Stein1990}. A big effort in this direction was provided by \cite{LT2009} with the Generalized Cauchy model, being also a decoupler of fractal dimension and Hurst effect.

Radial spectral densities are fundamental to spatial statistics. On the one hand, knowing at least local and global properties of a radial spectral density allows, by application of Tauberian theorems \cite{Stein1990}, to inspect the properties of the associated Gaussian field in terms  of mean square differentiability, fractal dimension, Hurst effect, and reproducing kernel Hilbert spaces \cite{porcu2012some}. On the other hand, the radial spectral density covers a fundamental part of statistical inference for Gaussian fields under infill asymptotics \cite{Bevilacqua_et_al:2016,Stein1990}. Finally, radial spectral densities are fundamental to inspect the so called screening effect, which in turns plays an important role in spatial prediction, and we refer the reader to \cite{stein2002screening} as well as to the recent work by \cite{porcu2020stein}.

An expression for the radial spectral density associated with the Dagum family has been elusive so far. A first attempt being made by \cite{Rossella2021}, who showed that such a spectral density admits a series expansion that is absolutely convergent.

This paper provides some insights in this direction. After providing background material in Section \ref{sec2}, Section \ref{sec3} provides the main results, which are classified into three parts: we start by deriving series expansions associated with the isotropic spectral density of the Dagum class. We then provide a closed form expression, in terms of the Fox--Wright functions, for such a class of isotropic spectral densities. Finally, we provide local and global asymptotic identities. Proofs are lengthy and technical: for a neater exposition, we deferred them to the Appendix. Section \ref{sec5} concludes the paper with a short discussion.

\section{Background Material}\label{sec2}

\subsection{Positive Definite Radial Functions}

We denote by $\{Z(\mathbf{s}), \mathbf{s} \in \mathbb{R}^d \} $ a centred Gaussian random field in $\mathbb{R}^d$, with the stationary covariance function $C:\mathbb{R}^d
\to \mathbb{R}$. We consider the class $\Phi_d$ of continuous mappings
$\phi:[0,\infty) \to \mathbb{R}$ with $\phi(0) =1$, such that
\[ %\label{generator}
{\rm cov} \left ( Z(\mathbf{s}), Z(\mathbf{s}^{\prime}) \right )= C(\mathbf{s}^{\prime}-\mathbf{s})=  \phi(\|\mathbf{s}^{\prime } -\mathbf{s} \|),
\]
with $\mathbf{s},\mathbf{s}^{\prime} \in \mathbb{R}^d$, and $\|\cdot\|$ denoting the Euclidean norm. Gaussian fields with such covariance functions are called \emph{weakly stationary and isotropic}. The function $C$ is called \emph{isotropic} or \emph{radially symmetric}, and the function $\phi$ its \emph{radial part}.
%We call $C$ the radial version of $\phi$ and $\phi$ its generator. In the following, we will be ambiguous when referring to $C$ or $\phi$ as covariance function, or covariance model, whenever no confusion can arise.

\cite{Shoe38} characterized the class $\Phi_d$ as being scale mixtures
of the characteristic functions of random vectors uniformly
distributed on the spherical shell of $\mathbb{R}^d$:
$$ \phi(r)= \int_{0}^{\infty} \Omega_{d}(r \xi) F(\mathrm{d} \xi), \qquad r \ge 0,$$
with $\Omega_{d}(r)= r^{-(d-2)/2}J_{(d-2)/2}(r)$ and $J_{\nu}$ a Bessel function of order $\nu$. Here, $F$ is a probability measure. The function $\phi$ is the uniquely determined characteristic function of a random vector, $\mathbf X$, such that $\mathbf X = \mathbf U \cdot R$, where equality is intended in the distributional sense, where $\mathbf U$ is uniformly distributed over the spherical shell of $\mathbb{R}^d$, $R$ is a nonnegative random variable with probability distribution, $F$, and where $\mathbf U$ and $R$ are independent.

\cite{daley-porcu} describes the properties of the measures, $F$, termed the \emph{Schoenberg measures} there, and shows the existence of projection operators that map the elements of $\Phi_d$ onto the elements of $\Phi_{d'}$, for $d' \ne d$. Throughout, we adopt their illustrative name and will call the function $F$ associated with $\phi$ a \emph{Schoenberg measure}. The derivative of $F$ is called the \emph{isotropic spectral density}. If $\phi$ is absolutely integrable, then the Fourier inversion (the Hankel transform) becomes possible. The Fourier transforms of radial versions of the members of $\Phi_d$, for a given $d$, have a simple expression, as reported in \cite{Stein:1999} and \cite{Yaglom:1987}. For a member $\phi$ of the family $\Phi_d$, we define its isotropic spectral density as
\[
 \widehat{\phi}(z)= \frac{z^{1-d/2}}{(2 \pi)^{d/2}} \int_{0}^{\infty} u^{d/2} J_{d/2-1}(uz)  \phi(u) \,{\rm d} u, \qquad z \ge 0.
\]
%{\bf \color{red} We should be more careful  with powers of $2\pi$ otherwise we will have problems in further formulas. Indeed, the Hankel transformations has its orifin in the Fourier transform of the radial functions.
%The calculation for the radial functions gives
%\[
% \widehat{F}(z)= \int_{\R^d}e^{-i \zz \cdot \rr}F(r)\text{d}\rr = z^{1-d/2}(2 \pi)^{\frac{d}{2}} \int_{0}^{\infty} z^{d/2} J_{d/2-1}(rz)  F(r) {\rm d} r
%\]
%However, we define the Fourier transform of the radial function like this
%\[
% \widehat{F}(z)= \frac{1}{(2\pi)^d}\int_{\R^d}e^{-i \zz \cdot \rr}F(r)\text{d}\rr.
%\]
%This means that the Hankel transformation should be defined by us like this
%} {\bf \color{blue}
% \[
% \widehat{\phi}(z)= \frac{z^{1-d/2}}{(2 \pi)^{\frac{d}{2}}} \int_{0}^{\infty} u^{d/2} J_{d/2-1}(uz)  \phi(u) {\rm d} u, \qquad z \ge 0. \]}
%
%\textbf{\color{magenta} For me, as a probabilist, a definition of the Fourier transform is correct if it coincides with the definition of the \emph{characteristic function} of a $d$-dimensional random vector. I attach an old file on the subject which I have written about 2 years ago for my colleague who is a physicist.}

The classes $\Phi_d$ are nested, with the inclusion relation $\Phi_{1} \supset \Phi_2 \supset \ldots \supset \Phi_{\infty}$ being strict, and where $\Phi_{\infty}:= \bigcap_{d \ge 1} \Phi_d$ is the class of mappings $\phi$ whose radial version is positive definite on all $d$-dimensional Euclidean spaces.

\subsection{Parametric Families of Isotropic Covariance Functions}

The Generalized Cauchy family of members of $\Phi_{\infty}$ \cite{Gneiting:Schlather:2004} is defined as:
\begin{equation} \label{g-cauchy}
{\cal C}_{\delta,\lambda}(r) = \left ( 1 + r^{\delta} \right )^{-\lambda/\delta}, \qquad r \ge 0,
\end{equation}
where the conditions $\delta \in (0,2]$ and $\lambda>0$ are necessary and sufficient for ${\cal C}_{\delta,\lambda}$ to belong to the class $\Phi_{\infty}$.
The parameter $\delta$ is crucial for the
differentiability at the origin and, as a consequence, for the
degree of  differentiability of the associated sample paths.
Specifically,  for $\delta=2$, they are infinitely  differentiable
 and they are not differentiable for $\delta \in (0,2)$.

For a Gaussian random field in $\mathbb{R}^d$ with isotropic covariance function ${\cal C}_{\delta,\lambda}(\|\cdot\|)$,
the sample paths have fractal dimension $D = d + 1-\delta/2$ for $\delta \in (0,2)$ and, if $\lambda \in (0, d]$,  the long memory parameter or Hurst coefficient is identically equal to $H = 1- \lambda/2$. Thus, $D$ and $H$ may vary independently of each other  \cite{Gneiting:Schlather:2004,LT2009}.   \cite{Tar_Mor2018} and \cite{LT2009} have shown that the isotropic spectral density, $\widehat{{\cal C}}_{\delta,\lambda}$ of the Generalized Cauchy covariance function is identically equal to
\[\label{spectrCG}
\widehat{{\cal C}}_{\delta,\lambda}(z) =-\frac{z^{-d}}{2^{d/2-1}\pi^{d/2+1}}\mathrm{Im}\int_0^\infty\frac{{\cal K}_{(d-2)/2}( t)}{(1+\exp(i\frac{\pi\delta}{2})(t/z)^{\delta})^{\lambda/\delta}}t^{d/2}\,\mathrm{d}t, \qquad  z \geq0,
\]
for  $\lambda>0$ and $\delta \in (0,2)$.

The Dagum family ${\cal D}_{\lambda,\delta} : [0,\infty) \to \mathbb{R}$ is defined as
\begin{equation} \label{averrohe}
{\cal D}_{\lambda,\delta}(r)=1-\left(\frac{r^\delta
}{1+r^\delta}\right )^{\lambda},
  \qquad r \ge 0.
\end{equation}
\cite{porcu-thesis} and subsequently \cite{porcu2007modelling} show that ${\cal D}_{\lambda,\delta}$ belongs to the class $\Phi_{3}$ provided $\delta < (7-\lambda)/(1+5\lambda)$ and $\lambda<7$.  \cite{berg2008} have shown that ${\cal D}_{\lambda,\delta} \in \Phi_{\infty}$ if and only if the function ${\cal A}_{\delta,\lambda}$, defined as
$$ {\cal A}_{\delta,\lambda}(r) = \frac{r^{\delta \lambda -1 }}{\left ( 1+ r^{\delta}\right )^{\lambda+1} }, \qquad r \ge 0, $$
belongs to $\Phi_{\infty}$. In particular, sufficient conditions for ${\cal D}_{\lambda,\delta} \in \Phi_{\infty}$ become $\delta \lambda \le 1$ and $\beta \le 1$. Also, for $\delta=1/\lambda$ we have ${\cal D}_{\lambda,1/\lambda} \in \Phi_{\infty}$ if and only if $\delta \le 1$.

%\begin{figure}[h!]
%\begin{tabular}{cc}
%  % after \\: \hline or \cline{col1-col2} \cline{col3-col4} ...
%   \includegraphics[width=5.2cm, height=6.5cm]{graf_1_d=1} & \includegraphics[width=5.2cm, height=6.5cm]{graf_1_d=2}\\
%   (A)&(B)\\
%\end{tabular}
%\caption{From left to right: (A)  comparison  of ${\cal D}_{\lambda,\delta}(\rr)$
%when $\delta=.7$ and  with different values of $\lambda=2.8,2,1,.4$.
%(B) comparison  of ${\cal D}_{\lambda,\delta}(\rr)$
%when $\lambda=.9$  and  with different values of $\delta=1.7,1.2,.7,.2$
%\label{fig:fcovt222}}
%\end{figure}
%

To simplify notation, throughout we shall write ${\cal D}_{\lambda,\delta}(\boldsymbol{r})$ for ${\cal D}_{\lambda,\delta} \circ \|\boldsymbol{r}\|$, $\boldsymbol{r} \in \mathbb{R}^d$, with $\circ$ denoting composition. Similar notation will be used for ${\cal C}_{\delta,\lambda}(\boldsymbol{r})$, $\boldsymbol{r} \in \mathbb{R}^d$. Analogously, we use $\widehat{{\cal D}}_{\delta,\lambda}(\boldsymbol{z})$ for $\widehat{{\cal D}}_{\delta,\lambda}(\|\boldsymbol{z}\|)$, $\boldsymbol{z} \in \mathbb{R}^d$, and sometimes we shall make use of the notation $z$ for $\|\boldsymbol{z}\|$. Similar notation will be used for the isotropic spectral density $\widehat{C}_{\delta,\lambda}$.

\subsection{Fractal Dimensions and the Hurst effect}

The local properties of a time series or a surface of $\mathbb{R}^d$ are identified through the so-called \emph{fractal dimension}, $D$, which is a roughness measure with range $[d,d+1)$, and with higher values indicating rougher surfaces.
The long memory in time series or spatial data is associated with power law correlations, and often referred to as the \emph{Hurst effect}. Long memory dependence is characterized by the $H$ parameter.
Local and global properties of a Gaussian random field have an intimate connection with its associated isotropic covariance function. In particular, if, for some $\alpha \in (0,2]$, the radial part $\varphi \in \Phi_d$ satisfies
\begin{equation}
\label{5}  \lim_{r \to 0} \frac{\varphi(r) }{r^{\alpha}} = 1,
\end{equation}
then the realizations of the Gaussian random field have fractal dimension $D=d+1- \alpha/2$, with probability $1$. Thus, estimation of $\alpha$ is linked with that of the fractal dimension $D$. Conversely, if for some $\beta \in (0,1)$,
\begin{equation}
\label{6} \lim_{r \to \infty} \varphi(r) r^{\beta} = 1,
\end{equation}
then the Gaussian random field is said to have \emph{long memory}, with Hurst coefficient $H=1-\beta/2$. For $H \in (1/2,1)$ or $H \in (0,1/2)$ the correlation is said to be respectively \emph{persistent} or \emph{anti-persistent}.
 In general, $D$ and $H$ are independent of each other, but under the assumption of self-affinity  they find an intimate connection in the well-known linear relationship
$ D + H = d + 1$ .
The Cauchy model behaves like (\ref{5}) for $\alpha = \delta \in (0,2]$ and like (\ref{6})  for  $\beta= \lambda \in (0,1)$. %{\bf \color{red} This may be seen from the formula  for the Cauchy correlation function}
For the reparameterized version ${\cal D}_{\lambda,\delta/\lambda} $, we have exactly the same result. For both models the local and global behaviour parameters may be estimated independently.

For $\delta\in (0,2)$, the Dagum covariance function can be rewritten as $${\cal D}_{\lambda,\delta}(\boldsymbol{r})=1-(1+1/\|\boldsymbol{r}\|^\delta)^{-\lambda},
  \qquad \boldsymbol{r}\in\mathbb{R}^d.$$
  When  $\|\boldsymbol{r}\|$ is {\em large}, the Dagum covariance function has the following asymptotic behavior
   $${\cal D}_{\lambda,\delta}(\boldsymbol{r})\sim\lambda\|\boldsymbol{r}\|^{-\delta}, \quad \text{for}\quad \delta\in (0,1). $$ Hence, under these parameter restrictions, a Gaussian random field has long
memory with Hurst coefficient $H = 1 - \delta/2$ with $\delta\in (0,1)$.

A notable fact is the following:
  \begin{equation}\label{short}
\int_{[0,\infty)^{d}}{\cal D}_{\lambda,\delta}(\boldsymbol{r})\text{d}^{d}\boldsymbol{r}=\frac{\pi^{d/2}}{2^{d-1}\Gamma(d/2)}\int_0^{\infty}r^{d-1}\left(1-\frac{1}{(1/r^{\delta}+1)^\lambda}\right)\text{d}r.
 \end{equation}
Furthermore,
  \[
r^{d-1}{\cal D}_{\lambda,\delta}(\boldsymbol{r})\sim r^{d-\delta-1},
  \qquad r\to \infty,\quad\text{with}\quad \|\boldsymbol{r}\|=r.
\]
The above implies that the integral (\ref{short}) is finite if $\delta>d$. An alternative way to see this is to notice that Equation (\ref{formula}) of this paper proves that
 \[
\begin{aligned}
\int_{\mathbb{R}^d}{\cal D}_{\lambda,\delta}(r)\text{d}^{d}\boldsymbol{r} &= \frac{\pi^{d/2}}{\Gamma(1+d/2)}\int_0^{\infty}r^{d-1}\left(1-\frac{1}{(1/r^{\delta}+1)^\lambda}\right)\text{d}r\\ &= \frac{\pi^{d/2}}{\Gamma(1+d/2)}
\frac{{\rm B}(d/\delta+\lambda,-d/\delta)}{\delta}.
\end{aligned}
 \]

\section{Theoretical Results}\label{sec3}

\subsection{Isotropic Spectral Density of the Dagum Covariance Function}

This subsection aims to compute the isotropic spectral density associated with the Dagum class in $\mathbb{R}^d$, for a given positive integer $d$, when
 $\delta>d$.  The spectral density of Dagum class can be written as
  \[\label{dagumSpec}
\widehat{\cal D}_{\delta,\lambda}(\boldsymbol{z})=\frac{1}{(2\pi)^d}\int_{\mathbb{R}^d}e^{-i \boldsymbol{z} \cdot \boldsymbol{r}}{\cal D}_{\lambda,\delta}(\boldsymbol{r})\,\mathrm{d}\boldsymbol{r},
\] with $i$ denoting the imaginary unit.
\cite{Rossella2021} showed that when $\lambda=k\in2\mathbb{Z}_{+}$ is an integer, the Dagum spectral density can be written as
  \[
\widehat{\cal D}_{\delta,k}(\boldsymbol{z})=-\displaystyle\sum_{h=0}^{k-1}(-1)^{k-h}\text{C}_h^k\widehat{\cal C}_{\delta,k-h}(\boldsymbol{z}), \qquad \boldsymbol{z}\in\mathbb{R}^d,
\]
where $\widehat{\cal C}_{\delta,k-h}$ is a generalized Cauchy isotropic spectral density associated with the generalized Cauchy function, ${\cal C}_{\delta,k-h}$ as defined at (\ref{g-cauchy}).

We start by extending this result for any $\lambda>0$ and for $d=1$.

\begin{theorem}\label{theod2}
For $d=1$ and $\delta>1$,  the Dagum isotropic spectral density $\widehat{D}_{\delta,\lambda}$ has the following explicit form:
\[
\widehat{\cal D}_{\delta,\lambda}(z)=-\frac{1}{\pi}\mathrm{Im}\int_0^\infty\left[1-\frac{e^{i\delta\lambda\pi/2}r^{\delta\lambda}}{(1+e^{i\pi/2}r^\delta)^\lambda}\right]e^{-z r}\,\mathrm{d}r.
\]
\end{theorem}

\vspace{0.5cm}
To extend this result to $\mathbb{R}^d$, for $d>1$, we start by
considering the case $\delta\leq d$, $\delta\lambda\in(0,2)$ and $d>1$.\\
We can show the following.
\begin{theorem} \label{thm3}
For $d> 1$, $\delta\in (0,2)$,  and $\delta\lambda\in (0,2]$, the Dagum isotropic spectral density $\widehat{D}_{\delta,\lambda}$ is given by
%\begin{equation}\label{SDI1}
%\widehat{\cal D}_{\delta,\lambda}(\zz)=-\frac{z^{-\nu}}{2^{\nu}\pi^{\nu+2}}\mathrm{Im}\int_0^\infty K_\nu(z t)\left(1-\frac{t^{\delta\lambda}e^{i\pi\delta\lambda/2}}{(1+e^{i\pi\delta/2}t^\delta)^\lambda}\right)t^{\nu+1}\,\mathrm{d}t,
%\end{equation}
%with $z=\norm{\zz}.$
%{\bf \color{red} As I said I propose to remove $\gamma$ completely. We do not use it at all. This is what should be written}
%{\bf \color{blue}
\begin{equation}\label{SDI1}
\widehat{\cal D}_{\delta,\lambda}(\boldsymbol{z})=-\frac{z^{1-\frac{d}{2}}}{2^{\frac{d}{2}-1}\pi^{\frac{d}{2}+1}}
\mathrm{Im}\int_0^\infty K_{\frac{d}{2}-1}(z t)\left(1-\frac{t^{\delta\lambda}e^{i\pi\delta\lambda/2}}{(1+e^{i\pi\delta/2}t^\delta)^\lambda}
\right)t^{d/2}\,\mathrm{d}t,
\end{equation}
with $z=\|\boldsymbol{z}\|.$
\end{theorem}

Albeit Theorems \ref{theod2} and \ref{thm3} provide some insight into understanding the isotropic spectral density associated with the Dagum covariance function, it might be desirable to have an explicit expression for such spectral densities. The following subsection attacks this problem.

%\vspace*{2cm}
\subsection{Dagum spectral density expressed as Fox--Wright function} \label{sec4}
%\subsection{New closed form of Dagum spectral density expressed as Fox--Wrigth function }
We start by defining the Fox--Wright function \cite{fox1928asymptotic, wright1935asymptotic} $_p\Psi_q$, through the identity
\begin{equation}\label{Fox}
%\begin{aligned}
_p\Psi_q \Big[\begin{matrix}
 (a_1,A_1) &, \cdots ,& ( a_p,A_p) \\
 (b_1,B_1)&, \cdots ,& (b_q,B_q) \end{matrix}
; -z \Big]
=\sum_{k=0}^\infty\frac{(-1)^k \Gamma(a_1+kA_1)\cdots\Gamma(a_p+kA_p)}{k!\Gamma(b_1+kB_1)\cdots\Gamma(b_q+kB_q)}z^{k}.
%\end{aligned}
\end{equation}
It turns out that this class of special functions is intimately related to the Dagum spectral density. We formally state this fact below.

\begin{theorem}\label{the3} Let $d$ be a positive integer, $\boldsymbol{z} \in \mathbb{R}^d$ and $z = \| \boldsymbol{z} \|$.
 Let ${\cal D}_{\lambda,\delta}$ be the Dagum class as defined at \emph{(\ref{averrohe})}.  For $\delta\in(0,2)$ and $\delta\lambda\in(0,2)$, the isotropic spectral density ${\cal D}_{\lambda,\delta}$ in $\mathbb{R}^d$, is given by

% {\bf \color{blue}
 \begin{equation}\label{spectrGW}
\begin{aligned}
\widehat{\cal D}_{\delta,\lambda}(\boldsymbol{z}) &= \delta^{(d)}(\boldsymbol{z})   -\frac{\pi^{-d/2}z^{-d}}{\Gamma(\lambda)}\,
_2\Psi_1 \Big[\begin{matrix}
 (\lambda,1) & ( d/2,-\delta/2) \\
 & (0,\delta/2) \end{matrix}
; -\left(\frac{z}{2}\right)^\delta \Big]\\
& - \frac{1}{2^d\pi^{d/2}} \frac{1}{\Gamma(\lambda)} \frac{2}{\delta}  \,_2\Psi_1 \Big[\begin{matrix}
 (\lambda+d/\delta,2/\delta) & (-d/\delta,-2/\delta) \\
 & (d/2,1) \end{matrix}
; -\left(\frac{z}{2}\right)^2 \Big].
\end{aligned}
\end{equation} \label{serie}
where $\delta^{(d)}(\boldsymbol{z})= (2\pi)^{-d} \underset{\epsilon\to 0}{\lim} \prod_{k=1}^d(\frac{1}{\epsilon}\xi_{\epsilon}(z_k))$, with $\xi_{\epsilon}$ being the  unit impulse function, and
$_2\Psi_1$ is the Fox--Wright function given by equation \emph{(\ref{Fox})}.

%\brown{Igor, I do not understand here: what is the difference between (\ref{spectrGW}) and (\ref{spectrGW1})?}
%{\bf \color{blue}
%We may write this solution in an equivalent form for any value of $z$ without using  the Dirac $\delta$ function,
%\begin{equation}\label{spectrGW_2-3}
%\begin{aligned}
%\widehat{\cal D}_{\delta,\lambda}(\zz) & = - \frac{z^{-d}}{\pi^{d/2}} \frac{1}{\Gamma(\lambda)}\sum_{n=1}^\infty \frac{(-1)^n\ \Gamma(\lambda+n)\Gamma(d/2-n\delta/2)}{n!\Gamma(n\delta/2)}\left(\frac{z}{2}\right)^{n\delta}\\
%& - \frac{1}{2^d\pi^{d/2}} \frac{1}{\Gamma(\lambda)} \frac{2}{\delta}\sum_{n=0}^\infty \frac{(-1)^n \Gamma(\lambda+(d+2n)/\delta)\Gamma(-(d+2n)/\delta)}{n!\Gamma(n+d/2)}\left(\frac{z}{2}\right)^{2n}.
%\end{aligned}
%\end{equation}
%}

\end{theorem}

\subsection{Asymptotic Properties of Dagum Spectral Density }
We finish with some theoretical resuls relating to the asymptotic behaviour of the Dagum isotropic spectral density.
\begin{theorem}\label{the4}
For all $\delta\in(0,2)$  and $\delta\lambda\in(0,2]$, the low frequency limit of the spectral density
$\widehat{\cal D}_{\delta,\lambda}$ is given by

%\begin{equation}\label{spectrAsin}
%\begin{enumerate}
%\item $\widehat{\cal D}_{\delta,\lambda}(\zz)\sim \frac{\lambda \Gamma(1+\nu-\frac{1}{2}\delta)}{2^{\nu+1-\delta}\pi^d\Gamma(\frac{1}{2}\delta)}z^{-2\nu-2+\delta}$,\qquad if  \quad $\delta\in(\nu+1/2,2\nu+2)$, $z \to 0$;
%\item $\widehat{\cal D}_{\delta,\lambda}(\zz)\sim \frac{2^{-(3\nu+2)}\pi^{2\nu+2}\Gamma(\delta-d-1)\Gamma(d-\delta+1+\lambda)}{\delta\Gamma(\nu+1)\Gamma(\lambda)}$, \qquad if \quad $\delta\in(2\nu+2,2)$, $z \to 0$.
%\end{enumerate}
%{\bf \color{red} As I said I propose to remove $\gamma$ completely. We do not use it at all. This is what should be written}
%{\bf \color{blue}
\begin{enumerate}
\item $\widehat{\cal D}_{\delta,\lambda}(\boldsymbol{z})\sim \frac{2^{-\delta}\lambda \Gamma(d/2-\delta/2)}{\pi^{\frac{d}{2}}\Gamma(\delta/2)}z^{\delta-d}$
 \quad if $\delta\in(\frac{d-2}{2},d)$, $z \to 0$;
\item $\widehat{\cal D}_{\delta,\lambda}(\boldsymbol{z})\sim  \frac{1}{\delta \pi^{\frac{d}{2}} 2^{d-1}\Gamma(d/2) }    \frac{\Gamma(-\mathrm{d}/\delta)\Gamma(d/\delta+\lambda)}{\Gamma(\lambda)}$ \qquad if \quad $\delta\in(d,2)$, $z \to 0$.
\end{enumerate}

%\item $\widehat{\cal D}_{\delta,\lambda}(\xx)\sim \frac{2^{-(3\nu+2)}\pi^{2\nu+2}}{\Gamma(\nu+1)}\int_0^\infty %t^{2\nu+1}\left(1-\frac{t^{\delta\lambda}}{(1+t^\delta)^\lambda}\right)\mathbf{d}t$, \qquad if $\delta\in(d,2)$.
%\end{equation}
\end{theorem}

\begin{theorem}\label{the5}
For $0<\delta<2$, and $0<\delta\lambda\leq2$ when  $z \to \infty$,
\begin{equation}\label{spectrAsin}
\begin{aligned}
&\widehat{\cal D}_{\delta,\lambda}(\boldsymbol{z})\\
&\quad\sim \frac{2^{\delta\lambda}z^{-d - \delta\lambda}}{\pi^{\frac{d}{2}}\Gamma(\lambda)} \,_3\Psi_2
\Big[\begin{matrix}
(\lambda,1) & ( \delta\lambda/2 + 1, \delta/2 ) & ((\delta\lambda+d)/2,\delta/2)\\
 & (\delta\lambda/2 , \delta/2 ) & (1-\delta\lambda/2,-\delta/2)
 \end{matrix}
; -\left(\frac{2}{z}\right)^{\delta} \Big]\\
&\quad\sim \frac{2^{\delta\lambda-1}\lambda\delta z^{-d - \delta\lambda}}{\pi^{\frac{d}{2}}}
\frac{ \Gamma(d/2 + \lambda\delta/2)}{\Gamma(1-\lambda\delta/2)}.
\end{aligned}
\end{equation}
\end{theorem}

\section{Conclusion}\label{sec5}

We have obtained the expressions for the isotropic spectral density related to the Dagum family. Our results can now be used in research related to (a) best optimal unbiased  linear prediction (kriging) under infill asymptotics when using the Dagum family. This in turn relies on equivalence of Gaussian measures and on the ratio between the correct and the misspecified spectral density \cite{Bevilacqua_et_al:2016}. While the Mat{\'e}rn covariance function has been already studied under this setting \cite{Zhang:2004}, the characterization of equivalence of Gaussian measures under the Dagum family has been elusive so far. Also, (b) knowing the form of the spectral density will be crucial to obtain the space-time spectral densities associated with covariance functions having a dynamical support depending on a Dagum radius, as detailed by \cite{porcu2020nonseparable}.

\section*{Acknowledgements}
%The research work conducted by Moreno Bevilacqua  was supported in
%part  by FONDECYT grant 1160280 and by  Iniciativa Cient\'ifica Milenio - Minecon Nucleo Milenio MESCD, Chile.
Partial support was provided by FONDECYT grant 1130647, Chile for
Emilio Porcu
and   by Millennium
Science Initiative of the Ministry
of Economy, Development, and
Tourism, grant "Millenium
Nucleus Center for the
Discovery of Structures in
Complex Data" for Emilio Porcu.
Partial support was provided   by
by FONDECYT grant 11200749, Chile for Tarik Faouzi
and by grant
Diubb 2020525 IF/R from the university of Bio Bio for Tarik Faouzi.
Igor Kondrashuk was supported in part by Fondecyt (Chile) Grants Nos. 1121030, by DIUBB (Chile) Grants Nos. GI 172409/C and 181409 3/R.

\appendix

\section{Proofs}

\noindent \begin{proof}[Proof of Theorem~\emph{\ref{theod2}}]
 \[
\widehat{\cal D}_{\delta,\lambda}(z)=\frac{1}{\pi}\mathrm{Re}\int_{0}^{\infty}e^{i|z|t}\left(1-\frac{t^{\delta\lambda}}{(1+t^{\delta})^\lambda}\right)\,\mathrm{d}t.
\]
Let $f(t;z)=e^{i|z|t}\left(1-\frac{t^{\delta\lambda}}{(1+t^{\delta})^\lambda}\right)$, and consider $D_\xi=\{t\in \mathbb{C};|t|\leq \xi, \mathrm{Re}(t)>0,\mathrm{Im}(t)>0\}$. Then, for  $\delta\in (0,2)$ and $\delta\lambda\in (0,2)$, $f$ is an analytic function on $D_\xi$. By the Cauchy integral formula,

 \[
\oint\limits_{\partial D_\xi}f(t;z)\,\mathrm{d}t=0,
\]
where $\partial D_\xi$ is a boundary of $D_\xi$, the  union  of three  components: the line segment $L_1$ along
the real axis from 0 to $\xi$, the arc $C_\xi$ of the circle   $|t|=\xi$ from $\xi$ to $i\xi$, and the line segment $L_2$
along the imaginary axis from $i\xi$ to $0$.

Next, for any $t$  on the arc $C_\xi=\{t\in \mathbb{C}; |t|=\xi\}$, there is a phase $\varphi \in [0,\pi/2]$ such that $t=\xi~e^{i\varphi}$. Then
\[
\begin{aligned}
f(t;z)&=f(\xi~e^{i\varphi};z)\\
&=e^{i|z|\xi~e^{i\varphi}}\left(1-\Big[\frac{\xi^{\delta}e^{i\delta\varphi}}{1+\xi^\delta e^{i\delta\varphi}}\Big]^\lambda\right).
\end{aligned}
\]
The last term of the above equality can be expressed as

   \[
1-\Big[\frac{\xi^{\delta}e^{i\delta\varphi}}{(1+\xi^\delta e^{i\delta\varphi})}\Big]^\lambda =1-\sum_{j=0}^{\infty}(-1)^j \frac{\Gamma(\lambda+j)}{\Gamma(\lambda)}(\xi^{-\delta} e^{-i\delta\varphi})^j,
\]

then

 \begin{equation}\label{IC2}
\oint\limits_{\partial D_\xi}f(t)\,\mathrm{d}t=-\sum_{j=1}^{\infty}(-1)^j \frac{\Gamma(\lambda+j)}{\Gamma(\lambda)}\xi^{-j\delta}\int_{0}^{\pi/2} e^{-i \delta\varphi j}e^{i |z|\xi e^{i\varphi}}\,\mathrm{d}\varphi.
\end{equation}

The integral presented in Equation (\ref{IC2}) is expressed as
\[
\begin{aligned}
 \int_{0}^{\pi/2} e^{-i \delta\varphi j}e^{i |z|\xi e^{i\varphi}}\,\mathrm{d}\varphi=& \frac{1}{i}\int_{0}^{\pi/2} e^{-i (j\delta+1)\varphi}e^{i |z|\xi e^{i\varphi}}\,\mathrm{d}e^{i\varphi}\\
 =&\frac{1}{i}\int_{C_1}e^{i|z|\xi \omega}\omega^{-(j\delta+1)}\,\mathrm{d}\omega
 =\frac{(|z|\xi)^{j\delta}}{i}\int_{C_{|z|\xi}}e^{i u}u^{-(j\delta+1)}\,\mathrm{d}u,
\end{aligned}
\]
where  $C_1$ is an arc of the circle  $|\omega|= 1$ from $1$ to $i$ and $C_{|z|\xi}$ is an arc of the circle  $|u|= |z|\xi$ from $|z|\xi$ to $i |z|\xi$.

Then,
\begin{eqnarray*}
\underset{\xi\to\infty}{\lim}\oint\limits_{\partial D_\xi}f(t;z)\,\mathrm{d}t = \underset{\xi\to\infty}{\lim}\int\limits_{ L_1}f(t;x)\,\mathrm{d}t + \underset{\xi\to\infty}{\lim}\int\limits_{ L_2}f(t;z)\,\mathrm{d}t
\\
+ i\sum_{j=1}^{\infty}(-1)^j \frac{\Gamma(\lambda+j)}{\Gamma(\lambda)}|z|^{j\delta}\underset{\xi\to\infty}{\lim}\int_{C_{|z|\xi}} e^{i u}u^{-(j\delta+1)}\,\mathrm{d}u=0.
\end{eqnarray*}
However, we may state that
\begin{equation}
\underset{\xi\to\infty}{\lim}\int_{C_{|z|\xi}} e^{i u}u^{-(j\delta+1)}\,\mathrm{d}u=0.
\end{equation}
Finally, $$\underset{\xi\to\infty}{\lim}\int\limits_{ L_1}f(t;x)\,\mathrm{d}t=-\underset{\xi\to\infty}{\lim}\int\limits_{ L_2}f(t;z)\,\mathrm{d}t.$$
%}

The last result implies
 \[
 \begin{aligned}
\widehat{\cal D}_{\delta,\lambda}(z)&=\frac{1}{\pi}\mathrm{Re}\int_0^\infty\left[1-\frac{r^{\delta\lambda}}{(1+r^\delta)^\lambda}\right]e^{i|z|r}\,\mathrm{d}r\\
&=-\frac{1}{\pi}\mathrm{Im}\int_0^\infty\left[1-\frac{e^{i\delta\lambda\pi/2}r^{\delta\lambda}}{(1+e^{i\pi/2}r^\delta)^\lambda}\right]e^{-|z|r}\,\mathrm{d}r.
\end{aligned}
\]
\end{proof}

\begin{proof}[Proof of Theorem~\emph{\ref{thm3}}]
With $\|\boldsymbol{z}\|=z$, we write,
%\[
%\begin{aligned}
%&\int_{\R^d}e^{i\rr\zz}\left(-\frac{z^{-\nu}}{2^{\nu}\pi^{\nu+2}}\mathrm{Im}\int_0^\infty K_\nu(zt)\left(1-\frac{t^{\delta\lambda}e^{i\pi\delta\lambda/2}}{(1+e^{i\pi\delta/2}t^\delta)^\lambda}\right)t^{\nu+1}\,\mathrm{d}t\right)\,\mathrm{d}\zz\\
%=&(2\pi)^{\nu+1}r^{-\nu}\int_0^\infty J_{\nu}(rz)z^{\nu+1}\left(-\frac{z^{-\nu}}{2^{\nu}\pi^{\nu+2}}\mathrm{Im}\int_0^\infty K_\nu(zt)\left(1-\frac{t^{\delta\lambda}e^{i\pi\delta\lambda/2}}{(1+e^{i\pi\delta/2}t^\delta)^\lambda}\right)t^{\nu+1}\,\mathrm{d}t\right)\,\mathrm{d}z\\
%=&-\frac{2r^{-\nu}}{\pi}\mathrm{Im}\int_0^\infty \left(1-\frac{e^{i\pi\delta\lambda/2}t^{\delta\lambda}}{(1+e^{i\pi\delta/2}t^{\delta})^\lambda}\right)t^{\nu+1}\int_0^\infty z K_\nu(zt)J_\nu(zt)\\
%=&-\frac{1}{i\pi}\int_{-\infty}^\infty \left(1-\frac{t^{\delta\lambda}e^{i\pi\delta\lambda/2}}{(1+e^{i\pi\delta/2}t^\delta)^\lambda}\right)\frac{t}{(r^2+t^2)}\mathrm{d}t.
%\end{aligned}
%\]
%{\bf \color{red} Better to remove $\nu$. Also small corrections in Eq. are done} {\bf \color{blue}
\[
\begin{aligned}
&\int_{\mathbb{R}^d}e^{i\boldsymbol{r}\boldsymbol{z}}\left(-\frac{z^{1-\frac{d}{2}}}{2^{\frac{d}{2}-1}\pi^{\frac{d}{2}+1 }}\mathrm{Im}\int_0^\infty K_{\frac{d}{2}-1}(zt)\left(1-\frac{t^{\delta\lambda}e^{i\pi\delta\lambda/2}}{(1+e^{i\pi\delta/2}
t^\delta)^\lambda}\right)t^{d/2}\,\mathrm{d}t\right)\,\mathrm{d}\boldsymbol{z}\\
&=(2\pi)^{\frac{d}{2}}r^{1-\frac{d}{2}}\int_0^\infty J_{\frac{d}{2}-1}(rz)z^{\frac{d}{2}}\left(-\frac{z^{1-\frac{d}{2}}}{2^{\frac{d}{2}-1}\pi^{\frac{d}{2}+1}}\mathrm{Im}\int_0^\infty
K_{\frac{d}{2}-1}(zt)\right.\\
&\quad\times\left.\left(1-\frac{t^{\delta\lambda}e^{i\pi\delta\lambda/2}}{(1+e^{i\pi\delta/2}t^\delta)^\lambda}\right)t^{d/2}\,\mathrm{d}t\right)\,\mathrm{d}z\\
&=-\frac{2r^{1-\frac{d}{2}}}{\pi}\mathrm{Im}\int_0^\infty \left(1-\frac{e^{i\pi\delta\lambda/2}t^{\delta\lambda}}{(1+e^{i\pi\delta/2}t^{\delta})^\lambda}\right)t^{d/2}\int_0^\infty z
K_{\frac{d}{2}-1}(zt)J_{\frac{d}{2}-1}(zr)~\mathrm{d}z\\
&=-\frac{1}{i\pi}\int_{-\infty}^\infty \left(1-\frac{t^{\delta\lambda}e^{i\pi\delta\lambda/2}}{(1+e^{i\pi\delta/2}t^\delta)^\lambda}\right)\frac{t}{(r^2+t^2)}\mathrm{d}t.
\end{aligned}
\] %}

When $\delta\in(0,2)$ and $\delta\lambda\in(0,2)$, the last integral is expressed as

$$-\frac{1}{i\pi}\int_{-\infty}^\infty \left(1-\frac{t^{\delta\lambda}e^{i\pi\delta\lambda/2}}{(1+e^{i\pi\delta/2}t^\delta)^\lambda}\right)\frac{t}{(r^2+t^2)}\,\mathrm{d}t
=1-\frac{r^{\delta\lambda}}{(1+r^\delta)^\lambda}.$$

%Therefore,  the results of Theorem \ref{thm3} is valid for $\delta<d$. \\
%\brown{I agree in that there is something strange here. We should specify this condition so that this integral is well defined.}
\end{proof}

%{\bf \color{red} The condition   $\delta<d$  should be completely removed and even no mentioned. It does not match the reality. We know $\widehat{\cal D}_{\delta,\lambda}(\zz)$
%is singular at $z=0$ if $\delta<d$ and finite if   $\delta>d$ }

\begin{proof}[Proof of Theorem~\emph{\ref{the3}}]
To find the explicit form of the Dagum spectral density, we use the Mellin--Barnes transform \cite{barnes1908new}  defined through the identity
\begin{eqnarray}\label{asser1}
\frac{1}{(1+x)^{\delta}} = \frac{1}{2 i\pi }\frac{1}{\Gamma(\delta)}
\oint_{\Lambda} ~x^u\Gamma(-u)\Gamma(\delta+u) \,\text{d}u,
\end{eqnarray}
here $\Gamma(\cdot)$ denotes the Gamma function. This representation is valid for any $x \in \mathbb{R}$. The contour $\Lambda$ contains the vertical line which passes between left and right poles in the complex plane $u$
from negative to positive imaginary infinity, and should be closed to the left in case $x>1$, and to the right complex infinity if $0<x<1.$

%{\bf \color{blue}
Applying Equation (\ref{asser1}), we obtain
\[
\begin{aligned}
\widehat{{\cal D}}_{\delta,\lambda}(\boldsymbol{z})=&\frac{1}{(2\pi)^d}\left[\int_{\mathbb{R}^d} {\rm e}^{ -i\boldsymbol{r}\boldsymbol{z}} \text{d}\boldsymbol{r}
-\frac{1}{2 i\pi}\frac{1}{\Gamma(\lambda)}\int_{\mathbb{R}^d}{\rm e}^{ -i\boldsymbol{r}\boldsymbol{z}}\oint_{\Lambda} \Gamma(-u) \Gamma(u+\lambda)r^{-u\delta} \text{d}u~ \text{d} \boldsymbol{r}\right] \\
=&\frac{1}{(2\pi)^d}\left[\int_{\mathbb{R}^d} {\rm e}^{ -i \boldsymbol{r}\boldsymbol{z}}\text{d}\boldsymbol{r}-\frac{1}{2 i\pi}\frac{1}{\Gamma(\lambda)}\oint_{\Lambda} \Gamma(-u)
\Gamma(u+\lambda)\int_{\mathbb{R}^d} {\rm e}^{ -i \boldsymbol{r}\boldsymbol{z}}r^{-u\delta} \text{d} \boldsymbol{r}~ \text{d}u\right]
\end{aligned}
\]
We now invoke the well known relationship \cite{allendes2013solution},
\begin{eqnarray*} \int_{\mathbb{R}^d}
{\rm e}^{i\boldsymbol{x}\boldsymbol{z}}\|\boldsymbol{r}\|^{-u\delta}\text{d} \boldsymbol{r}=\frac{2^{d-u\delta}\pi^{d/2}\Gamma(d/2-u\delta/2)}{\Gamma(u\delta/2)\|\boldsymbol{z}\|^{d-u\delta}}.
\end{eqnarray*}
%%and, by abuse of notation, we now write $x:=\norm{\xx}$. We have
Hence,
\begin{equation} \label{tonto-lava}
\begin{aligned}
\widehat{\cal D}_{\delta,\lambda}(\boldsymbol{z})&=  \delta^{(d)}(\boldsymbol{z}) -\frac{1}{(2\pi)^d}\frac{1}{\Gamma(\lambda)}\frac{1}{2\pi i }\oint_{\Lambda}
~ \frac{2^{d-u\delta}\pi^{d/2} \Gamma(d/2-u\delta/2)\Gamma(-u) \Gamma\left(u+\lambda\right)} {\Gamma(u\delta/2)}\frac{1}{{z}^{d - u\delta}} \text{d}u\\
&=    \delta^{(d)}(\boldsymbol{z}) - \frac{z^{-d}}{\pi^{d/2}} \frac{1}{2\pi i}\frac{1}{\Gamma(\lambda)}\oint_{\Lambda}  \frac{\Gamma(-u)\Gamma(u +\lambda) \Gamma(d/2-u\delta/2)}
{\Gamma(u\delta/2)}\left(\frac{z}{2}\right)^{u\delta}  \text{d}u.
\end{aligned}
\end{equation}
For any given value of $|z/2|,$ it is not relevant whether it is smaller or greater than $1$. In fact,   the contour might be closed to the right complex infinity.
The above series is convergent for any values of the variable $z$.
The functions $u \mapsto \Gamma(- u)$ and $u \mapsto \Gamma\left(d/2 -u\delta/2\right)$ contain poles in the complex plane, respectively when $-u = -n$, and when $d/2 -u\delta/2 = -n,$
$ n \in \mathbb{Z}_{+}$. Using this fact and through direct inspection we obtain that  the right hand side in (\ref{tonto-lava}) matches with (\ref{spectrGW_2}).
\begin{equation}\label{spectrGW_2}
\begin{aligned}
\widehat{\cal D}_{\delta,\lambda}(\boldsymbol{z}) &=   \delta^{(d)}(\boldsymbol{z}) -
\frac{z^{-d}}{\pi^{d/2}} \frac{1}{\Gamma(\lambda)}\sum_{n=0}^\infty \frac{(-1)^n\ \Gamma(\lambda+n)\Gamma(d/2-n\delta/2)}{n!\Gamma(n\delta/2)}\left(\frac{z}{2}\right)^{n\delta}\\
& - \frac{1}{2^d\pi^{d/2}} \frac{1}{\Gamma(\lambda)} \frac{2}{\delta}\sum_{n=0}^\infty \frac{(-1)^n \Gamma(\lambda+(d+2n)/\delta)\Gamma(-(d+2n)/\delta)}{n!\Gamma(n+d/2)}\left(\frac{z}{2}\right)^{2n},
\end{aligned}
\end{equation}

Next,  we invoke the expression of Fox--Wright function as in (\ref{Fox}).
In particular, using Equations (\ref{serie}) and (\ref{Fox})  we obtain a new form of Dagum spectral density:

\[
\begin{aligned}
\widehat{\cal D}_{\delta,\lambda}(\boldsymbol{z}) &=   \delta^{(d)}(\boldsymbol{z})   -\frac{\pi^{-d/2}z^{-d}}{\Gamma(\lambda)}\,
_2\Psi_1 \Big[\begin{matrix}
 (\lambda,1) & ( d/2,-\delta/2) \\
 & (0,\delta/2) \end{matrix}
; -\left(\frac{z}{2}\right)^\delta \Big]\\
& - \frac{1}{2^d\pi^{d/2}} \frac{1}{\Gamma(\lambda)} \frac{2}{\delta}  \,_2\Psi_1 \Big[\begin{matrix}
 (\lambda+d/\delta,2/\delta) & (-d/\delta,-2/\delta) \\
 & (d/2,1) \end{matrix}
; -\left(\frac{z}{2}\right)^2 \Big].
\end{aligned}
\] %}
\end{proof}

%{\bf \color{blue}
Theorem {\ref{the4}} states that the asymptotic $z \to 0$ for this Fourier transform  $\widehat{\cal D}_{\delta,\lambda}(\boldsymbol{z})$ of the Dagum correlation function should be a constant when $\delta > d.$
When we put $z \to 0$ in (\ref{spectrGW_2}), we may see that only the Dirac $\delta$ function, the $n=0$ term in the first sum of (\ref{spectrGW_2}) and the $n=0$ term in the second sum of (\ref{spectrGW_2})
remain non-zero if $\delta > d.$  However, the $n=0$ term in the second sum is the constant
\begin{equation} \label{n0term-sum2}
- \frac{1}{2^d\pi^{d/2}} \frac{1}{\Gamma(\lambda)} \frac{2}{\delta} \frac{\Gamma(\lambda+d/\delta)\Gamma(-d/\delta)}{\Gamma(d/2)},
\end{equation}
This is the same constant that stands in the formulation of Theorem  {\ref{the4}} that states the limit $z \to 0$ is smooth for $\delta > d$ and the function    $\widehat{\cal D}_{\delta,\lambda}(\boldsymbol{z})$ is continuous at
$z=0$ if $\delta > d.$ This means the Dirac $\delta$ function should be canceled with the $n=0$ term in the first sum of (\ref{spectrGW_2}) and for any $z$ the final form of the spectral density is
\begin{equation}\label{spectrGW_2-2}
\begin{aligned}
\widehat{\cal D}_{\delta,\lambda}(\boldsymbol{z}) & = - \frac{z^{-d}}{\pi^{d/2}} \frac{1}{\Gamma(\lambda)}\sum_{n=1}^\infty \frac{(-1)^n\ \Gamma(\lambda+n)\Gamma(d/2-n\delta/2)}{n!\Gamma(n\delta/2)}\left(\frac{z}{2}\right)^{n\delta}\\
& - \frac{1}{2^d\pi^{d/2}} \frac{1}{\Gamma(\lambda)} \frac{2}{\delta}\sum_{n=0}^\infty \frac{(-1)^n \Gamma(\lambda+(d+2n)/\delta)\Gamma(-(d+2n)/\delta)}{n!\Gamma(n+d/2)}\left(\frac{z}{2}\right)^{2n},
\end{aligned}
\end{equation}
This may be only if the $n=0$ term  in the first sum of (\ref{spectrGW_2}) is interpreted as the Dirac $\delta$ function in the the sense of distributions. %}

\begin{proof}[Proof of Theorem~\emph{\ref{the4}}]
The first point can be proved by direct construction. We have
% {\bf \color{blue}
\begin{equation}\label{density_spec}
\begin{aligned}
\widehat{\cal D}_{\delta,\lambda}(\boldsymbol{z})& = \frac{z^{1-\frac{d}{2}}}{(2\pi)^{\frac{d}{2}}}\int_0^{\infty}J_{\frac{d}{2}-1}(rz) r^{\frac{d}{2}}\left(1-\frac{r^{\delta\lambda}}{(1+r^\delta)^\lambda}\right)\,\mathrm{d}r. \\
\end{aligned}
\end{equation}%}
To find the low frequency behavior of $\widehat{\cal D}_{\delta,\lambda}(\boldsymbol{z})$, we make use of
Equation (\ref{density_spec}). A change of variable   of the type $y=r z$ shows that
 % {\bf \color{blue}
\begin{equation}\label{density_spec_bis}
\begin{aligned}
\widehat{\cal D}_{\delta,\lambda}(\boldsymbol{z})&=\frac{z^{-\frac{d}{2}}}{(2\pi)^\frac{d}{2}}\int_0^{\infty}J_{\frac{d}{2}-1}(y)\left(1-\frac{(y/z)^{\delta\lambda}}{(1+(y/z)^\delta)^\lambda}\right)(y/z)^{d/2}\text{d}y\\
&=\frac{z^{-d}}{(2\pi)^\frac{d}{2}}\int_0^{\infty}J_{\frac{d}{2}-1}(y)\left(1-\frac{(y/z)^{\delta\lambda}}{(1+(y/z)^\delta)^\lambda}\right)y^{d/2}\text{d}y
\end{aligned}
\end{equation} %}

We now invoke the identity %{\bf \color{blue}
\begin{equation}\label{density_spec_bbis}
\begin{aligned}
1-\frac{1}{(1+(z/y)^\delta)^\lambda} =-\sum_{j=1}^{\infty} \frac{(-1)^j}{j!}\frac{\Gamma(\lambda+j)}{\Gamma(\lambda)}(z/y)^{j\delta} \sim \lambda (z/y)^\delta, \qquad \hbox{as } z\to 0^{+},
\end{aligned}
\end{equation} %}
to obtain

% {\bf \color{blue}
\begin{equation}\label{density_spec_1}
\widehat{\cal D}_{\delta,\lambda}(\boldsymbol{z}) \sim \frac{\lambda z^{-d}}{(2\pi)^\frac{d}{2}}\int_0^{\infty}J_{\frac{d}{2}-1}(y)  \left(\frac{z}{y}\right)^\delta y^{d/2}\text{d}y =
\frac{\lambda z^{\delta-d}}{(2\pi)^\frac{d}{2}}\int_0^{\infty}J_{\frac{d}{2}-1}(y)   y^{\frac{d}{2}-\delta}\text{d}y.
\end{equation} %}
 Using \cite[14,6.651]{zwillinger2014table}, we find that if $d/2-1/2<\delta<d$

%  {\bf \color{blue}
 \begin{equation}\label{density_spec_2}
\widehat{\cal D}_{\delta,\lambda}(\boldsymbol{z})\sim \frac{2^{d/2-\delta}\lambda \Gamma(d/2-\delta/2)}{(2\pi)^{\frac{d}{2}}\Gamma(\delta/2)}z^{\delta-d}
= \frac{2^{-\delta}\lambda \Gamma(d/2-\delta/2)}{\pi^{\frac{d}{2}}\Gamma(\delta/2)}z^{\delta-d}.
\end{equation} % }

%{\bf \color{red} The calculation from the integration of the angular part off (this result) and MB transformation (23) bring the same result in this limit when $\delta<d$ in this case. Indeed, only the second term of the
%first sum in (23) should be taken into account and we have } {\bf \color{blue}
%\begin{equation}
%\widehat{\cal D}_{\delta,\lambda}(\zz) \sim  -\pi^{-d/2}z^{-d}\frac{(-1)\ \Gamma(\lambda+1)\Gamma(\frac{d-\delta}{2})}{\Gamma(\lambda)\Gamma(\frac{\delta}{2})}\left(\frac{z}{2}\right)^{\delta}
%\end{equation}}
%{\bf \color{red} and we observe complete coincidence with (28)}

This result coincides with $n=1$ term of the first sum of Eqs. (\ref{spectrGW_2})  and  (\ref{spectrGW_2-2}). We may conclude this limit $z \to 0$ is singular for the Dagum spectral density if
$\delta < d.$ The Dagum spectral density is not continuous at the point $z = 0$ under the condition $\delta < d.$ The limit $z \to 0$
corresponds to the behaviour of the Dagum correlation function at $r \to \infty$ and to integrability of its  Fourier transformation. In this case the Dagum spectral density is singular at $z=0$ because the integral
of the Fourier transformation is not convergent for $z=0$ under the condition  $\delta < d.$

We now prove the second point.  When $z\to 0^+$, the Bessel of the second kind can be expressed asymptotically as
 \begin{equation}
 J_{\nu}(r z)\sim \frac{(r z)^\nu}{2^\nu\Gamma(\nu+1)}.
 \end{equation}
% {\bf \color{blue}
Thus we may write in this limit
\begin{eqnarray*}
\widehat{\cal D}_{\delta,\lambda}(\boldsymbol{z}) = \frac{z^{1-\frac{d}{2}}}{(2\pi)^{\frac{d}{2}}}\int_0^{\infty}J_{\frac{d}{2}-1}(rz) r^{\frac{d}{2}}\left(1-\frac{r^{\delta\lambda}}{(1+r^\delta)^\lambda}\right)\text{d}r \\
\sim \frac{z^{1-\frac{d}{2}}}{(2\pi)^{\frac{d}{2}} } \int_0^{\infty}\frac{(r z)^{\frac{d}{2}-1}}{2^{\frac{d}{2}-1}\Gamma(d/2)}  r^{\frac{d}{2}}\left(1-\frac{r^{\delta\lambda}}{(1+r^\delta)^\lambda}\right)\text{d}r \\
= \frac{1}{\pi^{\frac{d}{2}} 2^{d-1}\Gamma(d/2) } \int_0^{\infty} r^{d-1}   \left(1-\frac{r^{\delta\lambda}}{(1+r^\delta)^\lambda}\right)\text{d}r.
\end{eqnarray*}

We make a change of variable  of the type $r^\delta=u$, and we find that, if $\delta>d$,
 \begin{equation}\label{density_spec_4}
\begin{aligned}
\widehat{\cal D}_{\delta,\lambda}(\boldsymbol{z}) &\sim  \frac{1}{\delta \pi^{\frac{d}{2}} 2^{d-1}\Gamma(d/2) } \int_0^\infty  u^{(d-\delta)/\delta}\left(1-\frac{u^{\lambda}}{(1+u)^\lambda}\right)\text{d}u\\
\end{aligned}
\end{equation}
This integral is a finite constant under the condition $\delta>d$ and may be found by the change of the variables
\[
\begin{aligned}
\tau = \frac{u}{1+u} \Rightarrow   u = \frac{\tau}{1-\tau}:
\end{aligned}
\]

\[
\begin{aligned}
\int_0^\infty  u^{(d-\delta)/\delta}\left(1-\frac{u^{\lambda}}{(1+u)^\lambda}\right)\text{d}u  &=  -\int_0^1 \frac{\tau^{\frac{d}{\delta} -1}}{(1-\tau)^{\frac{d}{\delta} + 1}}\left(1-\tau^\lambda\right)\text{d}\tau\\
&=  -\lim_{\varepsilon \to 0 }\int_0^1 \frac{\tau^{\frac{d}{\delta} -1}}{(1-\tau)^{\frac{d}{\delta} + 1 - \varepsilon}}\left(1-\tau^\lambda\right)\text{d}\tau\\
&= -\lim_{\varepsilon \to 0 }\left[{\rm B}\left(\frac{d}{\delta},-\frac{d}{\delta} + \varepsilon\right)  - {\rm B}\left(\frac{d}{\delta} + \lambda,-\frac{d}{\delta} +  \varepsilon\right)\right]
\end{aligned}
\]

\begin{equation}\label{formula}
\begin{aligned}
&= -\lim_{\varepsilon \to 0 }\left[\frac{\Gamma\left(\frac{d}{\delta}\right)}{\Gamma(\varepsilon)}  - \frac{\Gamma\left(\frac{d}{\delta}+\lambda\right)}{\Gamma(\varepsilon + \lambda)}\right]
\Gamma\left(-\frac{d}{\delta} + \varepsilon\right)\\
&=\frac{\Gamma\left(\frac{d}{\delta}+\lambda\right)\Gamma\left(-\frac{d}{\delta} \right)}{\Gamma(\lambda)}.
\end{aligned}
\end{equation}

Thus, we have
\[
 \begin{aligned}
 \widehat{\cal D}_{\delta,\lambda}(\boldsymbol{z}) \sim  \frac{1}{\delta \pi^{\frac{d}{2}} 2^{d-1}\Gamma(d/2) }    \frac{\Gamma(-\mathrm{d}/\delta)\Gamma(d/\delta+\lambda)}{\Gamma(\lambda)}.
\end{aligned}
\] %}

%{\bf \color{red} The calculation from the integration of the angular part off (this result) and MB transformation (23) bring the same result in this limit when $\delta>d$ in this case. Indeed, only the $n=0$ term of the
%second sum in (23) should be taken into account and we have } {\bf \color{blue}
%\begin{equation}
%\widehat{\cal D}_{\delta,\lambda}(\zz) \sim  - \frac{1}{2^d\pi^{d/2}} \frac{1}{\Gamma(\lambda)} \frac{2}{\delta}   \frac{\Gamma(\lambda +d/\delta)\Gamma(-d/\delta)}{\Gamma(d/2)},
%\end{equation}}
%{\bf \color{red} and we observe complete coincidence with (33)}

This result coincides with $n=0$ term  (\ref{n0term-sum2}) of the second sum of Eqs. (\ref{spectrGW_2})  and  (\ref{spectrGW_2-2}). We may conclude this limit $z \to 0$ is smooth
for the Dagum spectral density if
$\delta > d.$ The Dagum spectral density is continuous at the point $z = 0$ under the condition $\delta > d.$  The limit $z \to 0$
corresponds to the behaviour of the Dagum correlation function at $r \to \infty$ and to integrability of its  Fourier transformation. In this case the Dagum spectral density is not singular at $z=0$ because the integral
of the Fourier transformation is convergent for $z=0$ under the condition  $\delta > d.$

\end{proof}

\begin{proof}[Proof of Theorem~\emph{\ref{the5}}]

To find the high frequency behaviour of the Dagum spectral density, we need to use Equation (\ref{SDI1}). Indeed, as $z\to\infty$, we have

% {\bf \color{blue}

\begin{equation}
\begin{aligned}
\widehat{\cal D}_{\delta,\lambda}(\boldsymbol{z})&= -\frac{z^{1-\frac{d}{2}}}{2^{\frac{d}{2} -1}\pi^{\frac{d}{2} +1}}
\mathrm{Im}\int_0^\infty K_{\frac{d}{2}-1}(z t)\left(1-\frac{t^{\delta\lambda}e^{i\pi\delta\lambda/2}}{(1+e^{i\pi\delta/2}t^\delta)^\lambda}\right)t^{\frac{d}{2}}\,\mathrm{d}t\\
&= -\frac{z^{-d}}{2^{\frac{d}{2} -1}\pi^{\frac{d}{2} +1}}
\mathrm{Im}\int_0^\infty K_{\frac{d}{2}-1}(t)
\left(1-\frac{z^{-\delta\lambda}t^{\delta\lambda}e^{i\pi\delta\lambda/2}}{(1+e^{i\pi\delta/2}z^{-\delta}t^\delta)^\lambda}\right)
t^{\frac{d}{2}}\,\mathrm{d}t\\
&= \frac{z^{-d}}{2^{\frac{d}{2} -1}\pi^{\frac{d}{2} +1}}
\mathrm{Im}\int_0^\infty K_{\frac{d}{2}-1}(t)
(1 + z^{\delta} e^{-i\pi\delta/2} t^{-\delta})^{-\lambda}
t^{\frac{d}{2}}\,\mathrm{d}t\\
&= \frac{z^{-d}}{2^{\frac{d}{2} -1}\pi^{\frac{d}{2} +1}}
\mathrm{Im}\int_0^\infty K_{\frac{d}{2}-1}(t) z^{-\delta\lambda} e^{i\pi\lambda\delta/2} t^{\lambda\delta}
(1 + z^{-\delta} e^{i\pi\delta/2} t^{\delta})^{-\lambda}
t^{\frac{d}{2}}\,\mathrm{d}t\\
&= \frac{z^{-d}}{2^{\frac{d}{2} -1}\pi^{\frac{d}{2} +1}}
\mathrm{Im}\int_0^\infty K_{\frac{d}{2}-1}(t) z^{-\delta\lambda} e^{i\pi\lambda\delta/2} t^{\lambda\delta}
\sum_{k=0}^{\infty}\frac{(-1)^k \Gamma(\lambda+k)e^{i k\pi\delta/2}t^{k\delta}}{k! \Gamma(\lambda)z^{k\delta}}
t^{\frac{d}{2}}\,\mathrm{d}t\\
&= \frac{z^{-d - \delta\lambda}}{2^{\frac{d}{2} -1}\pi^{\frac{d}{2} +1}}
\mathrm{Im}\displaystyle\sum_{k=0}^{\infty}\frac{(-1)^k \Gamma(\lambda+k)e^{i (k+\lambda)\pi\delta/2}}{k!
\Gamma(\lambda)z^{k\delta}}\int_0^\infty K_{\frac{d}{2}-1}(t)t^{(k+\lambda)\delta+\frac{d}{2}}\,\mathrm{d}t. \\
\end{aligned}
\end{equation}
 Using \cite[6.561,16]{zwillinger2014table}, we obtain
\[
\begin{aligned}
\widehat{\cal D}_{\delta,\lambda}(\boldsymbol{z})
& = \frac{z^{-d - \delta\lambda}}{\pi^{\frac{d}{2} +1}}
\mathrm{Im}\displaystyle\sum_{k=0}^{\infty}\frac{(-1)^k 2^{(k+\lambda)\delta}\Gamma(\lambda+k)\Gamma(\frac{d+(k+\lambda)\delta}{2})
\Gamma(1+\frac{(k+\lambda)\delta}{2})e^{i (k+\lambda)\pi\delta/2}}{k! \Gamma(\lambda)z^{k\delta}}\\
& = \frac{z^{-d - \delta\lambda}}{\pi^{\frac{d}{2} +1}}
\displaystyle\sum_{k=0}^{\infty}\frac{(-1)^k 2^{(k+\lambda)\delta}\Gamma(\lambda+k)\Gamma(\frac{d+(k+\lambda)\delta}{2})
\Gamma(1+\frac{(k+\lambda)\delta}{2})\sin{(k+\lambda)\pi\delta/2}}{k! \Gamma(\lambda)z^{k\delta}}.
\end{aligned}
\]

 We use  Euler\textquotesingle s formula given by
\[
\sin(k\delta\pi/2)=\frac{\pi}{\Gamma(k\delta/2)\Gamma(1-k\delta/2)},
\]
and  we write
  \begin{equation}
\begin{aligned}
\widehat{\cal D}_{\delta,\lambda}(\boldsymbol{z})& = \frac{2^{\delta\lambda}z^{-d - \delta\lambda}}{\pi^{\frac{d}{2}}\Gamma(\lambda)}
\displaystyle\sum_{k=0}^{\infty}\frac{(-1)^k \Gamma(\lambda+k)\Gamma(d/2 + (k+\lambda)\delta/2)\Gamma(1+(k+\lambda)\delta/2)}{k!\Gamma((k+\lambda)\delta/2)\Gamma(1-(k+\lambda)\delta/2)}\left(\frac{2}{z}\right)^{k\delta}\\
&= \frac{2^{\delta\lambda}z^{-d - \delta\lambda}}{\pi^{\frac{d}{2}}\Gamma(\lambda)} \,_3\Psi_2
\Big[\begin{matrix}
(\lambda,1) & ( \delta\lambda/2 + 1, \delta/2 ) & ((\delta\lambda+d)/2,\delta/2)\\
 & (\delta\lambda/2 , \delta/2 ) & (1-\delta\lambda/2,-\delta/2)
 \end{matrix}
; -\left(\frac{2}{z}\right)^{\delta} \Big].
\end{aligned}
\end{equation}

However, in this limit we should leave only the first term of this sum
\begin{equation} \label{asymp}
\begin{aligned}
\widehat{\cal D}_{\delta,\lambda}(\boldsymbol{z})& \sim \frac{2^{\delta\lambda-1}\lambda\delta z^{-d - \delta\lambda}}{\pi^{\frac{d}{2}}}
\frac{ \Gamma(d/2 + \lambda\delta/2)}{\Gamma(1-\lambda\delta/2)}
\end{aligned}
\end{equation}
% }

The proof is completed.

 %_2\Psi_1 \Big[\begin{matrix}
 %(\lambda,1) & ( \nu+1+\delta\lambda/2 , \delta/2 ) \\
 %& (1-\delta\lambda/2 , -\delta/2 ) \end{matrix}
%; -x^{-\delta} \Big]\right]
%\end{equation}
\end{proof}

%{\bf \color{blue}
At the end, we would like to comment that the same result for the asymptotic (\ref{asymp}) of the spectral density may be found by transforming the integrand of the contour integral  (\ref{tonto-lava})
to another integrand such that the convergence of the series which arises in the result of the residues calculus when we close  the contour of the integral to the left complex infinity
becomes obvious because this series would satisfy the  standard criteria of convergence. Such a form of the integrand is useful to study the asymptotic of the spectral density at $z \to \infty$
and would allow us to reproduce the asymptotical result (\ref{asymp}).%}

%\bibliographystyle{ECA_jasa} 
%bibliographystyle{apalike}
\bibliographystyle{apt}
\bibliography{template}
\end{document}